\documentclass[12pt, reqno]{amsart}
\usepackage{amsmath, amsthm, amscd, amsfonts, amssymb, graphicx, color}
\usepackage[bookmarksnumbered, colorlinks, plainpages]{hyperref}
\hypersetup{colorlinks=true,linkcolor=red, anchorcolor=green, citecolor=cyan, urlcolor=red, filecolor=magenta, pdftoolbar=true}

\usepackage{mathrsfs}% To use \mathscr
\newtheorem{theorem}{Theorem}[section]
\newtheorem{lemma}[theorem]{Lemma}
\newtheorem{proposition}[theorem]{Proposition}

\theoremstyle{definition}
\newtheorem{definition}[theorem]{Definition}

\theoremstyle{remark}
\newtheorem{remark}[theorem]{Remark}
\numberwithin{equation}{section}

\begin{document}

\setcounter{page}{1}

\title[Hausdorff operators on modulation and Wiener amalgam spaces]{Hausdorff operators on modulation and Wiener amalgam spaces}

\author[G. Zhao, D. Fan, \MakeLowercase{and} W. Guo]{Guoping Zhao,$^1$ Dashan Fan,$^2$ \MakeLowercase{and} Weichao Guo$^3$$^{*}$}

\address{$^{1}$School of Applied Mathematics, Xiamen University of Technology, Xiamen, 361024, P.R.China}
\email{\textcolor[rgb]{0.00,0.00,0.84}{guopingzhaomath@gmail.com}}

\address{$^{2}$Department of Mathematics, University of Wisconsin-Milwaukee, Milwaukee, WI 53201, USA}
\email{\textcolor[rgb]{0.00,0.00,0.84}{fan@uwm.edu}}

\address{$^{3}$School of Mathematics and Information Sciences, Guangzhou University, Guangzhou, 510006, P.R.China}
\email{\textcolor[rgb]{0.00,0.00,0.84}{weichaoguomath@gmail.com}}

%\dedicatory{This paper is dedicated to Professor ABCD}

\let\thefootnote\relax\footnote{Copyright 2016 by the Tusi Mathematical Research Group.}



\date{Received: xxxxxx; Revised: yyyyyy; Accepted: zzzzzz.
\newline \indent $^{*}$Corresponding author}

\begin{abstract}
  We give the sharp conditions for boundedness of Hausdorff operators on certain modulation and Wiener amalgam spaces.
\end{abstract} \maketitle

\section{Introduction and Preliminary}
Hausdorff operator, originated from some classical summation methods, has a long history in the study of real and complex analysis.
We refer the reader to \cite{Chen_Fan_Wang_2013} and \cite{Liflyand_survey} for a survey of some historic background and recent developments
about Hausdorff operator.

For a suitable function $\Phi$, one of the corresponding Hausdorff operator $H_{\Phi}$ can be defined by
\begin{equation}\label{Definition, Hausdorff operator}
  H_{\Phi}f(x)=\int_{\mathbb{R}^n}\Phi(y)f\left(\frac{x}{|y|}\right)dy.
\end{equation}
Although there is a general definition where $f(A(y)x)$, with matrix $ A$,
stays in place of $f(x/|y|)$ in (\ref{Definition, Hausdorff operator}), we only consider the special case in this paper.
However we do not exclude that the general case will prove to be of interest as well and we are keep interested in the general case.

There are many known results about the boundedness of Hausdorff operators on various function spaces,
such as \cite{Lerner-Liflyand_Austr., Liflyand_Acta_Sci_Math, Liflyand-Miyachi_Studia,Liflyand-Moricz_Proceeding}.
Unfortunately, the sharp conditions on boundedness of Hausdorff operator can be characterized in only few cases.
One can see \cite{Wu_Chen_SCIChina_2014} for the sharp characterization for the boundedness of Hausdorff operators on $L^p$,
and see \cite{Fan_Lin_Analysis_2014, Ruan_Fan_JMAA_2016} for the sharp characterization for the boundedness of Hausdorff operators on Hardy spaces $H^1$ and $h^1$.
We observe that
the characterizations of boundedness of Hausdorff operator were also established in some other function spaces
(see \cite{Chen_Fan_Wang_2013,Gao-Zhao_Anal.Math_2015}). However,
we find that these spaces have some similar properties as the $L^{p}$ spaces. Let us briefly describe this fact in the following.

In order to prove the necessity of boundedness of Hausdorff operator on $L^p$,
we must choose a suitable function $f$ and estimate $\|H_{\Phi}f\|_{L^p}$ from below by some integral regarding $\Phi$.
The space $L^p$ is fit for this lower estimates,
since for a function $f$, the norm $\|f\|_{L^p}$ only depends on the absolute value of $f$
and the $L^p$ norm has the scaling property $\|f(s\cdot)\|_{L^p}=s^{-n/p}\|f\|_{L^p}$.
%$\|f(\lambda\cdot)\|_{L^p}=\lambda^{-n/p}\|f\|_{L^p}$.
We observe that the function spaces, for which the characterizations of boundedness of Hausdorff operator are established so far,
all have the above two properties as the $L^p$ spaces so that the proof of the necessity follows the same line as that on $L^p$.
However, in the case of frequency decomposition spaces, such as the modulation spaces or Wiener amalgam spaces,
the situation becomes quite different and complicated.

The modulation spaces $M_{p,q}^{s}$ were first introduced by Feichtinger \cite{Feichtinger} in 1983.
As function spaces associated with the uniform decomposition (see \cite{Tribel_modulation_1983}),
modulation spaces have a close relationship to the topic of time-frequency analysis (see \cite{Grochenig}),
and they have been regarded as appropriate function spaces for the study of partial differential equations (see \cite{Wang_Hudzik_JDE_2007}).
We refer the reader to \cite{Feichtinger_Survey} for
some motivations and historical remarks.
One can also refer our recent paper \cite{Guo-Wu-Yang-Zhao_JFA,Guo-Wu-Zhao_Proc.AMS}  for the properties of modulation spaces and Wiener amalgam spaces.

As a frequency decomposition space, the norm of $f$ in a modulation space can not be completely determined by the absolute value of the function.
On the other hand, the scaling property of modulation spaces is not as simple as that of $L^p$(see \cite{Sugimoto_Tomita_JFA_2007}).
Thus, it is interesting to find out the sharp conditions for the boundedness of Hausdorff operator on modulation spaces,
since in this case the method used in the $L^p$ case is not adoptable.

We also consider the boundedness of Hausdorff operator on Wiener amalgam spaces $W_{p,q}^s$.
In general, a Wiener amalgam space can be represented by $W(B,C)$,
where $B$ and $C$ are served as the local and global component respectively.
In this paper, we consider a special case $W(\mathscr{F}^{-1}L_q^s,L_p)$, which is closely related to modulation spaces.
For simplicity in the notation, we also use $W_{p,q}^s$ to denote this function space.
Before stating the main theorems, we make some preparations as follows.

We need to add some suitable assumptions on $\Phi$.
Firstly, in order to establish the sharp conditions for the boundedness of Hausdorff operator, we
assume $\Phi\geqslant 0$.
In the proof of the necessity part, we must make some (pointwise) estimates from below.
That is why the assumption $\Phi\geqslant 0$ is necessary in most of the known characterizations for
the boundedness of Hausdorff operator on function spaces(see \cite{Fan_Lin_Analysis_2014,Ruan_Fan_JMAA_2016,Wu_Chen_SCIChina_2014}).

Secondly, we make another assumption for $\Phi$ as following:
\begin{equation}\label{basic assumption for Phi}
  \int_{B(0,1)}|y|^n\Phi(y)dy<\infty,\\
  \text{ and }
  \int_{B(0,1)^c}\Phi(y)dy<\infty.
\end{equation}

We would like to give following remarks not only for explaining the reasonability of the assumption (\ref{basic assumption for Phi}),
but also to give some important properties of Hausdorff operator under the assumption (\ref{basic assumption for Phi}).
\begin{remark}[Assumption (\ref{basic assumption for Phi}) is weakest]
In fact, (\ref{basic assumption for Phi}) is the weakest assumption to ensure that the Schwartz function can be mapped into tempered distribution by
Hausdorff operator $H_{\Phi}$.

On one hand, if $H_{\Phi}f\in \mathscr{S}'$, it must be locally integrable,
and since $\Phi\geqslant0$, we have
\begin{equation*}
  \begin{split}
    \infty
    >
     &
    \int_{B(0,1)}|H_{\Phi}f(x)|dx
    =
    \int_{B(0,1)}\int_{\mathbb{R}^n}\Phi(y)f(x/|y|)dydx
    \\
    = &
    \int_{\mathbb{R}^n}\Phi(y)\int_{B(0,1)}f(x/|y|)dx dy
    %\\
    = %&
    \int_{\mathbb{R}^n}\Phi(y)|y|^n\int_{B(0,\frac{1}{|y|})}f(x)dx dy
    \\
    = &
    \int_{B(0,1)}\Phi(y)|y|^n\int_{B(0,\frac{1}{|y|})}f(x)dx dy + \int_{B^c(0,1)}\Phi(y)|y|^n\int_{B(0,\frac{1}{|y|})}f(x)dx dy.
  \end{split}
\end{equation*}
On the other hand,
for any nonnegative Schwartz function $f$ satisfying $f=1$ on $B(0,1)$, we have
\begin{equation*}
  \begin{split}
    \int_{B(0,1)}|H_{\Phi}f(x)|dx
    \geqslant &
    \int_{B(0,1)}\Phi(y)|y|^n\int_{B(0,1)}f(x)dx dy + \int_{B^c(0,1)}\Phi(y)|y|^n\int_{B(0,\frac{1}{|y|})}dxdy
    \\
    \sim &
    \int_{B(0,1)}\Phi(y)|y|^n dy + \int_{B^c(0,1)}\Phi(y) dy.
  \end{split}
\end{equation*}
This implies that
\begin{equation*}
  \begin{split}
    \int_{B(0,1)}\Phi(y)|y|^ndy+ \int_{B^c(0,1)}\Phi(y) dy<\infty.
  \end{split}
\end{equation*}
\end{remark}

\begin{remark}[$H_{\Phi}f$ is well defined as a tempered distribution]
  If $\Phi$ satisfies (\ref{basic assumption for Phi}),
$H_{\Phi}f$ makes sense for all $f\in \mathscr{S}(\mathbb{R}^n)$ for the reason that for $x\neq 0$,
\begin{equation*}
  \begin{split}
    |H_{\Phi}f(x)|
    \leqslant &
    \left(\int_{B(0,1)}
    +
    \int_{B^c(0,1)}\right)\Phi(y)f(x/|y|)dy
    \\
    \leqslant &
    |x|^{-n}\int_{B(0,1)}|y|^n\Phi(y)(|x/|y||^{n}f(x/|y|))
    +
    \|f\|_{L^{\infty}}\int_{B^c(0,1)}\Phi(y)dy
    \\
    \lesssim &
    C_f(1+|x|^{-n})\left(\int_{B(0,1)}|y|^n\Phi(y)dy+\int_{B^c(0,1)}\Phi(y)dy\right)<\infty,
  \end{split}
\end{equation*}
and
\begin{equation*}
  \begin{split}
    &\int_{B(0,1)}|H_{\Phi}f(x)|dx
    \leqslant
    \int_{B(0,1)}\int_{\mathbb{R}^n}\Phi(y)|f(x/|y|)|dydx
    \\
    &=
    \int_{B(0,1)}\int_{B(0,1)}\Phi(y)|f(x/|y|)|dydx
    +
    \int_{B(0,1)}\int_{B^c(0,1)}\Phi(y)|f(x/|y|)|dydx
    \\
    &\leqslant
    \int_{B(0,1)}\Phi(y)\int_{\mathbb{R}^n}|f(x/|y|)|dx dy
    +
    |B(0,1)|\cdot\|f\|_{L^{\infty}}\cdot\int_{B^c(0,1)}\Phi(y)dy
    \\
    &\leqslant
    \int_{B(0,1)}|y|^n\Phi(y)dy \|f\|_{L^1}
    +
    |B(0,1)|\cdot\|f\|_{L^{\infty}}\cdot\int_{B^c(0,1)}\Phi(y)dy<\infty.
  \end{split}
\end{equation*}
Thus, for $f\in \mathscr{S}(\mathbb{R}^n)$, $H_{\Phi}f$ is a locally integrable function, which has polynomial growth at infinity.
It implies that $H_{\Phi}f$ is a tempered distribution for $f\in \mathscr{S}(\mathbb{R}^n)$.
Write
\begin{equation*}
  \langle H_{\Phi}f, g\rangle=\int_{\mathbb{R}^n}H_{\Phi}f(x)g(x)dx,
\end{equation*}
where $\langle u,f\rangle$ means the action of a tempered distribution $u$ on a Schwartz function $f$.
\end{remark}

\begin{remark}[$H_{\Phi}: \mathscr{S}\rightarrow \mathscr{S}'$ is continuous]
For $f,g\in \mathscr{S}(\mathbb{R}^n)$, we have that
\begin{equation*}
  \int_{\mathbb{R}^n}|f(x/|y|)g(x)|dx
  \leqslant
  \|f\|_{L^{\infty}}\|g\|_{L^1}
\end{equation*}
and
\begin{equation*}
  \int_{\mathbb{R}^n}|f(x/|y|)g(x)|dx
  \leqslant
  \|g\|_{L^{\infty}}\|f(\cdot/|y|)\|_{L^1}
  \leqslant |y|^n\|g\|_{L^{\infty}}\|f\|_{L^1}.
\end{equation*}
It follows that
\begin{equation*}
  \begin{split}
    &\int_{\mathbb{R}^n}|\Phi(y)|\int_{\mathbb{R}^n}|f(x/|y|)g(x)|dxdy
    \\
    \lesssim &
    (\|f\|_{L^1}+\|f\|_{L^{\infty}})(\|g\|_{L^1}+\|g\|_{L^{\infty}})
    \int_{\mathbb{R}^n}|\Phi(y)|\min\{1,|y|^n\}dy.
  \end{split}
\end{equation*}
Thus,
\begin{equation*}
  \begin{split}
    |\langle H_{\Phi}f, g\rangle|
    = &
    |\int_{\mathbb{R}^n}H_{\Phi}f(x) g(x)dx|
    \\
    \leqslant &
    \int_{\mathbb{R}^n}\int_{\mathbb{R}^n}\Phi(y)|f({x}/{|y|})|dy g(x)dx
    \\
    \leqslant &
    \int_{\mathbb{R}^n}\Phi(y)\int_{\mathbb{R}^n}|f({x}/{|y|})|\cdot|g(x)|dxdy
    \\
    \lesssim &
    (\|f\|_{L^1}+\|f\|_{L^{\infty}})(\|g\|_{L^1}+\|g\|_{L^{\infty}})
    \int_{\mathbb{R}^n}|\Phi(y)|\min\{1,|y|^n\}dy.
  \end{split}
\end{equation*}
Using the definition of Schwartz function space, we have
$|<H_{\Phi}f, g_l>|\rightarrow 0$, for $f,g_l\in \mathscr{S}(\mathbb{R}^n)$
satisfying that $g_l\rightarrow 0$ as $l\rightarrow\infty$ in the topology of $\mathscr{S}$.
\end{remark}

\begin{remark}[Fourier transform of $H_{\Phi}f$]
  Define
\begin{equation*}
  \widetilde{H_{\Phi}}f(x)=\int_{\mathbb{R}^n}\Phi(y)|y|^nf(|y|x)dy.
\end{equation*}
By a similar method used before, we can verify that
$\widetilde{H_{\Phi}}f$ is a tempered distribution and that the map $\widetilde{H_{\Phi}}: \mathscr{S}\rightarrow \mathscr{S}'$ is continuous.

Moreover, we have
\begin{equation*}
  \widehat{H_{\Phi}f}=\widetilde{H_{\Phi}}\widehat{f} \hspace{6mm}\text{\  in the distribution sense.}
\end{equation*}
Indeed, for $f,g\in \mathscr{S}(\mathbb{R}^n)$, we have
\begin{equation*}
  \begin{split}
    \langle\widehat{H_{\Phi}f},g\rangle
    = &
    \langle H_{\Phi}f,\widehat{g}\rangle
    %\\
    = %&
    \int_{\mathbb{R}^n}\int_{\mathbb{R}^n}\Phi(y)f({x}/{|y|})dy \widehat{g}(x)dx
    \\
    = &
    \int_{\mathbb{R}^n}\Phi(y)\int_{\mathbb{R}^n}f({x}/{|y|}) \widehat{g}(x)dx dy
    \\
    = &
    \int_{\mathbb{R}^n}\Phi(y)\int_{\mathbb{R}^n}\widehat{f({\cdot}/{|y|})}(x)g(x)dx dy
    \\
    = &
    \int_{\mathbb{R}^n}\Phi(y)\int_{\mathbb{R}^n}|y|^n\widehat{f}(|y|x)g(x)dx dy
    \\
    = &
    \int_{\mathbb{R}^n}\int_{\mathbb{R}^n}|y|^n\Phi(y)\widehat{f}(|y|x)dyg(x)dx=\langle\widetilde{H_{\Phi}}\widehat{f},g\rangle.
  \end{split}
\end{equation*}
\end{remark}

\begin{remark}[Adjoint operator of $H_{\Phi}f$]
  We define the complex inner product
  \begin{equation*}
    \langle f|\ g\rangle=\int_{\mathbb{R}^n}f(x)\overline{g}(x)dx.
  \end{equation*}
  The adjoint operator of $H_{\Phi}f$ is defined by
  \begin{equation*}
    \langle H_{\Phi}f|\ g\rangle=\langle f|\ H_{\Phi}^*g\rangle
  \end{equation*}
  for $f,g\in \mathscr{S}(\mathbb{R}^n)$.
  By a direct calculation, we have
  \begin{equation*}
    \begin{split}
      \langle H_{\Phi}f|\ g\rangle
      = &
      \int_{\mathbb{R}^n}H_{\Phi}f(x) \overline{g}(x)dx
      %\\
      = %&
      \int_{\mathbb{R}^n}\int_{\mathbb{R}^n}\Phi(y)f\left({x}/{|y|}\right)dy\overline{g}(x)dx
      \\
      = &
      \int_{\mathbb{R}^n}\Phi(y)\int_{\mathbb{R}^n}f\left({x}/{|y|}\right)\overline{g}(x)dx dy
      %\\
      = %&
      \int_{\mathbb{R}^n}|y|^n\Phi(y)\int_{\mathbb{R}^n}f(x)\overline{g}(|y|x)dx dy
      \\
      = &
      \int_{\mathbb{R}^n}f(x)\int_{\mathbb{R}^n}\Phi(y)|y|^n\overline{g}(|y|x)dy dx
      =
      \langle f|\ \widetilde{H_{\Phi}}g\rangle.
    \end{split}
  \end{equation*}
  It follows that
  \begin{equation*}
  H_{\Phi}^*g=\widetilde{H_{\Phi}}g \hspace{6mm}\text{\  in the distribution sense.}
  \end{equation*}
\end{remark}

We turn to give definitions of modulation and Wiener amalgam spaces.

Let $\mathscr{S}:= \mathscr{S}(\mathbb{R}^{n})$ be the Schwartz space
and $\mathscr{S}':=\mathscr{S}'(\mathbb{R}^{n})$ be the space of tempered distributions.
We define the Fourier transform $\mathscr {F}f$ and the inverse Fourier transform $\mathscr {F}^{-1}f$ of $f\in \mathscr{S}(\mathbb{R}^{n})$ by
$$
\mathscr {F}f(\xi)=\hat{f}(\xi)=\int_{\mathbb{R}^{n}}f(x)e^{-2\pi ix\cdot \xi}dx
,
~~
\mathscr {F}^{-1}f(x)=f^{\vee}(x)=\int_{\mathbb{R}^{n}}f(\xi)e^{2\pi ix\cdot \xi}d\xi.
$$

The translation operator is defined as $T_{x_0}f(x)=f(x-x_0)$ and
the modulation operator is defined as $M_{\xi}f(x)=e^{2\pi i\xi \cdot x}f(x)$, for $x, x_0, \xi\in\mathbb{R}^n$.
Fixed a nonzero function $\phi\in \mathscr{S}$, the short-time Fourier
transform of $f\in \mathscr{S}'$ with respect to the window $\phi$ is given by
\begin{equation*}
V_{\phi}f(x,\xi)=\langle f,M_{\xi}T_x\phi\rangle,
\end{equation*}
and that can be written as
\begin{equation*}
  V_{\phi}f(x,\xi)=\int_{\mathbb{R}^n}f(y)\overline{\phi(y-x)}e^{-2\pi iy\cdot \xi}dy
\end{equation*}
if $f\in \mathscr{S}$.
We give the (continuous) definition of modulation space $\mathcal {M}_{p,q}^s$ as follows.

\begin{definition}
Let $s \in \mathbb{R}$, $0<p,q\leqslant \infty$. The (weighted) modulation space $\mathcal {M}_{p,q}^s$ consists
of all $f\in \mathscr{S}'(\mathbb{R}^n)$ such that the (weighted) modulation space norm
\begin{equation*}
\begin{split}
\|f\|_{\mathcal {M}_{p,q}^{s}}&=\big\|\|V_{\phi}f(x,\xi)\|_{L_{x,p}}\big\|_{L_{\xi,q}^s}
\\&
=\left(\int_{\mathbb{R}^n}\left(\int_{\mathbb{R}^n}|V_{\phi}f(x,\xi)|^{p}dx\right)^{{q}/{p}}\langle \xi\rangle^{sq}dx\right)^{{1}/{q}}
\end{split}
\end{equation*}
is finite, with the usual modifications when $p=\infty$ or $q=\infty$.
This definition is independent of the choice of the window $\phi\in \mathscr{S}$.
\end{definition}
Applying the frequency-uniform localization techniques, one can give an alternative definition of modulation spaces (see \cite{Tribel_modulation_1983}
%, Wang_Hudzik_JDE_2007
 for details).

We denote by $Q_{k}$ the unit cube with the center at $k$. Then the family $\{Q_{k}\}_{k\in\mathbb{Z}^{n}}$
constitutes a decomposition of $\mathbb{R}^{n}$.
Let $\eta \in \mathscr {S}(\mathbb{R}^{n}),$
$\eta: \mathbb{R}^{n} \rightarrow [0,1]$ be a smooth function satisfying that $\eta(\xi)=1$ for
$|\xi|_{\infty}\leqslant {1}/{2}$ and $\eta(\xi)=0$ for $|\xi|\geqslant 3/4$. Let
\begin{equation*}
\eta_{k}(\xi)=\eta(\xi-k),  k\in \mathbb{Z}^{n}
\end{equation*}
be a translation of \ $\eta$.
Since $\eta_{k}(\xi)=1$ in $Q_{k}$, we have that $\sum_{k\in\mathbb{Z}^{n}}\eta_{k}(\xi)\geqslant 1$
for all $\xi\in\mathbb{R}^{n}$. Denote
\begin{equation*}
\sigma_{k}(\xi)=\eta_{k}(\xi)\left(\sum_{l\in\mathbb{Z}^{n}}\eta_{l}(\xi)\right)^{-1},  ~~~~ k\in\mathbb{Z}^{n}.
\end{equation*}
It is easy to know that $\{\sigma_{k}\}_{k\in\mathbb{Z}^{n}}$
constitutes a smooth partition of the unity, and $%
\sigma_{k}(\xi)=\sigma(\xi-k)$. The frequency-uniform decomposition
operators can be defined by
\begin{equation*}
\Box_{k}:= \mathscr{F}^{-1}\sigma_{k}\mathscr{F}
\end{equation*}
for $k\in \mathbb{Z}^{n}$.
Now, we give the (discrete) definition of modulation space $M_{p,q}^s$.

\begin{definition}
Let $s\in \mathbb{R}, 0<p,q\leqslant \infty$. The modulation space $M_{p,q}^s$ consists of all $f\in \mathscr{S}'$ such that the (quasi-)norm
\begin{equation*}
\|f\|_{M_{p,q}^s}:=\left( \sum_{k\in \mathbb{Z}^{n}}\langle k\rangle ^{sq}\|\Box_k f\|_{p}^{q}\right)^{1/q}
\end{equation*}
is finite. We write $M_{p,q}:=M_{p,q}^0$ for short.
We also recall that this definition is independent of the choice of $\{\sigma_k\}_{k\in\mathbb{Z}^n}$ and the definitions of $\mathcal {M}_{p,q}^s$ and $M_{p,q}^s$ are equivalent \cite{Wang_Hudzik_JDE_2007}.
\end{definition}

\begin{definition}\label{Definition, Wiener amalgam space, continuous form}
Let $0<p, q\leqslant \infty$, $s\in \mathbb{R}$.
Given a window function $\phi\in \mathscr{S}\backslash\{0\}$, the Wiener amalgam space $W_{p,q}^s$ consists
of all $f\in \mathscr{S}'(\mathbb{R}^n)$ such that the norm
\begin{equation*}
\begin{split}
\|f\|_{W_{p,q}^{s}}&=\big\|\|V_{\phi}f(x,\xi)\|_{L_{\xi, q}^s}\big\|_{L_{x,p}}
\\&
=\left(\int_{\mathbb{R}^n}\left(\int_{\mathbb{R}^n}|V_{\phi}f(x,\xi)|^{q}\langle \xi\rangle^{sq}d\xi\right)^{{p}/{q}}dx\right)^{{1}/{p}}
\end{split}
\end{equation*}
is finite, with the usual modifications when $p=\infty$ or $q=\infty$. We write $W_{p,q}:=W_{p,q}^0$ for short.
\end{definition}

Now, we state our main results as follows.

\begin{theorem}\label{theorem, boundedness of Hausdorff operator on modulation space}
  Let $1\leqslant p,q\leqslant \infty$, $(1/p-1/2)(1/q-1/p)\geqslant 0$,
  $\Phi$ be a nonnegative function satisfying the basic assumption (\ref{basic assumption for Phi}).
  Then
  $H_{\Phi}$ is bounded on $M_{p,q}$
  if and only if
  \begin{equation*}
    \int_{\mathbb{R}^n}\left( |y|^{n/{p}}+|y|^{n/q'} \right) \Phi(y)dy<\infty.
  \end{equation*}
\end{theorem}

\begin{theorem}\label{theorem, boundedness of Hausdorff operator on Wiener amalgam space}
  Let $1\leqslant p,q\leqslant \infty$, $(1/q-1/2)(1/q-1/p)\leqslant 0$,
  $\Phi$ be a nonnegative function satisfying the basic assumption (\ref{basic assumption for Phi}).
  Then
  $H_{\Phi}$ is bounded on $W_{p,q}$
  if and only if
  \begin{equation*}
    \int_{\mathbb{R}^n}\left( |y|^{n/{p}}+|y|^{n/q'} \right) \Phi(y)dy<\infty.
  \end{equation*}
\end{theorem}

Our paper is organized as follows.
In Section 2, we collect some basic properties of modulation and Wiener amalgam spaces and give
 the proof of Theorem \ref{theorem, boundedness of Hausdorff operator on modulation space} and \ref{theorem, boundedness of Hausdorff operator on Wiener amalgam space}.

Throughout this paper, we will adopt the following notations.
We use $X\lesssim Y$ to denote the statement that $X\leqslant CY$, with a positive constant $C$ that may depend on $n, \,p$,
but it might be different from line to line.
The notation $X\sim Y$ means the statement $X\lesssim Y\lesssim X$.
We use $X\lesssim_{\lambda}Y$ to denote $X\leqslant C_{\lambda}Y$,
meaning that the implied constant $C_{\lambda}$ depends on the parameter $\lambda$.
For a multi-index $k=(k_{1},k_{2},...,k_{n})\in \mathbb{Z}^{n}$,
we denote $|k|_{\infty }:=\max\limits_{i=1,2,...,n}|k_{i}|$ and $\langle k\rangle:=(1+|k|^{2})^{{1}/{2}}.$

\section{Proof of main theorem }
First,
we list some basic properties about modulation spaces as follows.
\begin{lemma}[Symmetry of time and frequency]\label{lemma, relation of Mpp Wpp}
  $\|\mathscr{F}^{-1}f\|_{M_{p,q}}\sim\|f\|_{W_{q,p}}\sim\|\mathscr{F}f\|_{M_{p,q}}$.
\end{lemma}
\begin{proof}
  By the fact that
  \begin{equation*}
    |V_{\phi }f(x,\xi )|=|V_{\hat{\phi}}\hat{f}(\xi ,-x)|,
  \end{equation*}
  the conclusion follows by the definition of modulation and Wiener amalgam space.
\end{proof}

\begin{lemma}[\cite{Sugimoto-Tomita-JFA-2007} Dilation property of modulation space] \label{lemma, dilation property of modulation space}
  Let $1\leqslant p,q\leqslant \infty$, $(1/p-1/2)(1/q-1/p)\geq 0$.
  Set $f_{\lambda}(x)=f(\lambda x)$.
  Then
  \begin{equation*}
    \|f_{\lambda}\|_{M_{p,q}}\lesssim \max\{\lambda^{-n/p}, \lambda^{-n/{q'}}\}\|f\|_{M_{p,q}}.
  \end{equation*}
\end{lemma}

\begin{lemma}[\cite{Kobayashi_Sugimoto_JFA_2011} Embedding relations between modulation and Lebesgue spaces]\label{lemma, embedding relations between Mpp and Lp}
The following embedding relations are right:
\begin{enumerate}
  \item $M_{p,q}\hookrightarrow L^p$ for $1/q\geqslant1/p\geqslant1/2$;
  \item $L^p\hookrightarrow M_{p,q}$ for $1/q\leqslant 1/p\leqslant 1/2$.
\end{enumerate}
\end{lemma}

\begin{lemma}[\cite{Cunanan_Kobayashi-Sugimoto_JFA_2015} Embedding relations between Wiener amalgam and Lebesgue spaces]\label{lemma, embedding relations between Wpp and Lp}
The following embedding relations are right:
\begin{enumerate}
  \item $W_{p,q}\hookrightarrow L^p$ for $1/p\geqslant1/q\geqslant1/2$;
  \item $L^p\hookrightarrow W_{p,q}$ for $1/p\leqslant 1/q\leqslant 1/2$.
\end{enumerate}
\end{lemma}

\begin{lemma}\label{Holder inequality of modulation space and Wiener space}
  Let $1\leqslant p,q \leqslant\infty$. We have
  \begin{enumerate}
    \item
    \begin{equation*}
      \begin{split}
        \left|\int_{\mathbb{R}^n}f(x)\bar{g}(x)dx\right|
        \leqslant
        \|f\|_{M_{p',q'}}\|g\|_{M_{p,q}}
      \end{split}
    \end{equation*}
    \item
    \begin{equation*}
      \begin{split}
        \left|\int_{\mathbb{R}^n}f(x)\bar{g}(x)dx\right|
        \leqslant
        \|f\|_{W_{q',p'}}\|g\|_{W_{q,p}}.
      \end{split}
    \end{equation*}
  \end{enumerate}
\end{lemma}
\begin{proof}
  By Lemma \ref{lemma, relation of Mpp Wpp}, we only give the proof of the first inequality.
  Denote $\eta_k^*=\sum\limits_{l\in\mathbb{Z}^n:\eta_k\eta_l\neq0}\eta_k$ and $\Box_k^*=\mathscr{F}^{-1}\eta_k^*\mathscr{F}$.
  By the definition of modulation spaces and Plancherel's equality we get that
  \begin{equation*}
    \begin{split}
      \left|\int_{\mathbb{R}^n}f(x)\bar{g}(x)dx\right|
      &=\left|\int_{\mathbb{R}^n}\hat{f}\hat{g}d\xi\right|
      =\left|\int_{\mathbb{R}^n}\sum_{k\in\mathbb{Z}^n}\sigma_k \hat{f}\cdot \sum_{l\in\mathbb{Z}^n}\sigma_l\hat{g}dx\right|
      \\
      &=\left|\int_{\mathbb{R}^n}\sum_{k\in\mathbb{Z}^n}\sigma_k \hat{f}\cdot \sigma_k^*\hat{g}dx\right|
      =\left|\int_{\mathbb{R}^n}\sum_{k\in\mathbb{Z}^n}\Box_k f \cdot \Box_k^* gdx\right|
      \\
      &\leqslant
      \sum_{k\in\mathbb{Z}^n}\int_{\mathbb{R}^n}\left|\Box_k f\cdot \Box_k^*g\right|dx
      \leqslant
      \sum_{k\in\mathbb{Z}^n} \|\Box_k f\|_{L^{p'}}\|\Box_k^* g\|_{L^{p}}
      \\
      &\leqslant
      \left( \sum_{k\in\mathbb{Z}^n} \|\Box_k f\|_{L^{p'}}^{q'} \right)^{1/{q'}} \left( \sum_{k\in\mathbb{Z}^n}\|\Box_k^* g\|_{L^{p}}^q \right)^{1/q}
      \leqslant
      \|f\|_{M_{p',q'}}\|g\|_{M_{p,q}},
    \end{split}
  \end{equation*}
  where we use H\"{o}lder's inequality in the last two and third inequality and the fact that the definition of modulation space is independent of the decomposition function.
\end{proof}

In order to make the proof more clear, we give the following technical proposition.
\begin{proposition}[For technique]\label{proposition, for technique}
  Let $1/2\leqslant 1/p \leqslant 1/q\leqslant 1$, and $\Phi$ be a nonnegative function satisfying the basic assumption (\ref{basic assumption for Phi}).
  Then
  \begin{enumerate}
    \item if $H_{\Phi} : M_{p,q} \rightarrow L^p$ is bounded, we have
    \begin{equation*}
     \int_{\mathbb{R}^n}|y|^{n/p}\Phi(y)dy<\infty;
    \end{equation*}
    \item if $H_{\Phi}^* :W_{q,p} \rightarrow L^q$ is bounded, we have
    \begin{equation*}
     \int_{\mathbb{R}^n}|y|^{n/q'}\Phi(y)dy<\infty;
    \end{equation*}
    \item if $H_{\Phi}^* : M_{p,q} \rightarrow L^p$ is bounded, we have
    \begin{equation*}
     \int_{\mathbb{R}^n}|y|^{n/p'}\Phi(y)dy<\infty;
    \end{equation*}
    \item if $H_{\Phi} :W_{q,p} \rightarrow L^q$ is bounded, we have
    \begin{equation*}
     \int_{\mathbb{R}^n}|y|^{n/q}\Phi(y)dy<\infty.
    \end{equation*}
  \end{enumerate}

\end{proposition}
\begin{proof}
We only give the proof of statements (1) and (2), since the other cases can be handled similarly.

  Suppose $H_{\Phi} : M_{p,q} \rightarrow L^p$ is bounded.
  Let $\psi: \mathbb{R}^n \rightarrow [0,1]$ be a smooth bump function supported in the ball $\{\xi: |\xi|<\frac{3}{2}\}$
  and be equal to 1 on the ball $\{\xi: |\xi|\leqslant \frac{4}{3}\}$.
  Let $\rho(\xi)=\psi(\xi)-\psi(2\xi)$. Then $\rho$ is a positive smooth function
  supported in the annulus $\{\xi: \frac{2}{3} < |\xi| < \frac{3}{2} \}$,
  satisfying $\rho(\xi)=1$ on a smaller annulus $\{\xi: \frac{3}{4} \leqslant |\xi| \leqslant \frac{4}{3} \}$.
  Denote $\rho_j(\xi):=\rho(\xi/2^j)$.
  We have
  $\text{supp}\rho_j\subset\{\xi: \frac{2}{3} \cdot 2^j \leqslant |\xi| \leqslant \frac{3}{2} \cdot 2^j\}$
  and $\rho_j(\xi)=1$ on $\{\xi: \frac{3}{4} \cdot 2^j \leqslant |\xi| \leqslant \frac{4}{3} \cdot 2^j\}$.
  Thus, we have $\text{supp}\sum\limits_{j=1}^{N}\rho_j(\xi)\subset \{\xi: \frac{4}{3}\leqslant |\xi| \leqslant \frac{3}{2} \cdot 2^N \}$
  and $\sum\limits_{j=1}^{N}\rho_j(\xi)=1$ on $\{\xi: \frac{3}{2}\leqslant |\xi| \leqslant \frac{4}{3} \cdot 2^N \}$.

  Take $\varphi$ to be a nonnegative smooth function satisfying that $\text{supp}\widehat{\varphi}\subset B(0,1/2)$, $\varphi(0)=1$.
  Choose $f_{N}(x)=\left(\sum\limits_{j=0}^{N+1}\rho_j(x)\cdot|x|^{-n/p}\right)\ast \varphi$.
  So we have
  \begin{equation}\label{to control the number of Box-k}
   \text{supp}\widehat{f_N}\subset B(0,1/2)\text{ and } f_{N}(x)\gtrsim  \sum\limits_{j=1}^{N}\rho_j(x)\cdot|x|^{-n/p}.
  \end{equation}
  In above, the previous inclusion relation follows from the support condition of $\widehat{\varphi}$.
  We interpret the latter inequality.
  We only need to prove it when the right hand is nonzero, i.e. $x\in\{ \frac{4}{3}\leqslant |x| \leqslant \frac{3}{2} \cdot 2^N \}$.
  For the nonnegative function $\varphi$ satisfying $\varphi(0)=1$, there exists a positive constant $\delta<\min\{4/3,1/12\}$, such that $\varphi(x)>1/2$ when $ |x|<\delta$.
  By the triangle inequality and the properties of $\varphi$ we have that
  \begin{equation*}
    \begin{split}
      f_N(x)
      &=\left(\sum\limits_{j=0}^{N+1}\rho_j(x)\cdot|x|^{-n/p}\right)\ast \varphi
      %\\
      =
      \int_{\mathbb{R}^n} \left(\sum_{j=0}^{N+1}\rho_j(x-y)\cdot|x-y|^{-n/p}\right) \varphi(y)dy
      \\
      &\gtrsim
      \int_{\substack{\frac{4}{3}\leq|x-y|\leq \frac{3}{2} \cdot 2^N \\|y|<\delta}}
      \sum_{j=0}^{N+1}\rho_j(x-y) \cdot |x|^{-n/p} dy
      %\\
      \gtrsim
      \sum_{j=1}^{N}\rho_j(x)\cdot |x|^{-n/p},
    \end{split}
  \end{equation*}
  so we prove (\ref{to control the number of Box-k}).

%  Recalling that $\Phi$ is nonnegative, for $0<M<<N$,
%  we have the following estimate:
  \begin{equation*}
    \begin{split}
      \left\| H_{\Phi}f_N \right\|_{L^p}
      &=\left\| \int_{\mathbb{R}^n}\Phi(y)f_N(x/{|y|})dy \right\|_{L^p}
      \\
      &\gtrsim
      \left\| \int_{\mathbb{R}^n}\Phi(y)|y|^{n/p} \cdot \sum_{j=1}^{N}\rho_j(x/|y|) \cdot |x|^{-n/p}dy \right\|_{L^p}
      \\
      &\geqslant
      \left\| \int_{B(0,\frac{2}{3}\cdot2^M)\setminus B(0,\frac{3}{4}\cdot2^{-M})}\Phi(y)|y|^{n/p} \cdot
      \sum_{j=1}^{N}\rho_j(x/|y|) \cdot |x|^{-n/p}dy \right\|_{L^p}
      \\
      &\geqslant
      \left\| \int_{B(0,\frac{2}{3}\cdot2^M)\setminus B(0,\frac{3}{4}\cdot2^{-M})}\Phi(y)|y|^{n/p} \cdot
      \chi_{\{2^M<|x|< 2^{N-M}\}}(x)\cdot |x|^{-n/p}dy \right\|_{L^p}
      \\
      &=
      \int_{B(0,\frac{2}{3}\cdot2^M)\setminus B(0,\frac{3}{4}\cdot2^{-M})}\Phi(y)|y|^{n/p}dy
      \cdot \left\| |x|^{-n/p}\chi_{\{2^M<|x|< 2^{N-M}\}}(x) \right\|_{L^p}
      \\
      &\gtrsim
      \int_{B(0,\frac{2}{3}\cdot2^M)\setminus B(0,\frac{3}{4}\cdot2^{-M})}\Phi(y)|y|^{n/p}dy \cdot \left(\lg2^{N-2M}\right)^{1/p},
    \end{split}
  \end{equation*}
  where we use the fact that $\sum_{j=1}^{N}\rho_j(x/|y|)=1$ for $y\in B(0,\frac{2}{3}\cdot2^M)\setminus B(0,\frac{3}{4}\cdot2^{-M})$
  and  $x\in B(0,2^M)\setminus B(0,2^{N-M})$.
  On the other hand, observing that $\text{supp}\widehat{f_N}\subset B(0,1/2)$, we have that
  \begin{equation*}
    \begin{split}
      \left\| f_N \right\|_{M_{p,q}}
      &=
      \left(\sum_{\substack{\sigma_k\widehat{f_N}\neq 0\\k\in\mathbb{Z}^n}}\left\| \mathscr{F}^{-1}(\sigma_k \widehat{f_N}) \right\|_{L^p}^q\right)^{1/q}
      %\\
      %&
      \lesssim
      \left(\sum_{\substack{\sigma_k\widehat{f_N}\neq 0\\k\in\mathbb{Z}^n}}\left\|f_N \right\|_{L^p}^q\right)^{1/q}
      \\
      &\lesssim
      \|f_N\|_{L^p}\lesssim \|\sum\limits_{j=0}^{N+1}\rho_j(x)\cdot|x|^{-n/p}\|_{L^p}\sim (\ln2^N)^{1/p}.
    \end{split}
  \end{equation*}
  Using the boundedness of $H_{\Phi}$ and the above estimates for $H_{\Phi}f_N$ and $f_N$, we have that
  \begin{equation*}
    \|H_{\Phi}\|_{M_{p,q}\rightarrow L^p}
    \geqslant
    \frac{\left\| H_{\Phi}f_N\right\|_{L^p}}{\|f_N\|_{M_{p,q}}}
    \gtrsim
    \int_{B(0,\frac{2}{3}\cdot2^M)\setminus B(0,\frac{3}{4}\cdot2^{-M})}\Phi(y)|y|^{n/p}dy
    \left(\frac{\lg2^{N-2M}}{\lg2^{N}}\right)^{1/p}.
  \end{equation*}
  Letting $N\rightarrow\infty$, we have
  \begin{equation*}
    \int_{B(0,\frac{2}{3}\cdot2^M)\setminus B(0,\frac{3}{4}\cdot2^{-M})}\Phi(y)|y|^{n/p}dy
    \lesssim
    \|H_{\Phi}\|_{M_{p,q}\rightarrow L^p}.
  \end{equation*}
  By the arbitrariness of $M$, we let $M\rightarrow \infty$ and obtain that
  $\int_{\mathbb{R}^n}\Phi(y)|y|^{n/p}dy\lesssim \|H_{\Phi}\|_{M_{p,q}\rightarrow L^p}$.

  %%%%%%%%%%%%%%%%%%%%%%%%%%%%%%%%%%%%%%%%%%%%%%%%%%%%%%%%%%%%%%%%%%%%%%%%%%%%%%%%%%%%%%%%%%%%%%%%%%%%%%%%%%%%%%%%%%%%%%%%%%%%%%%%%%%%%%%%%%%%%%%%%%%
  %%%%%%%%%%%       proof of the second one           %%%%%%%%%%%%%%%%%%%%%%%%%%%%%%%%%%%%%%%%%%%%%%%%%%%%%%%%%%%%%%%%%%%%%%%%%%%%%%%%%%%%%%%%%%%%%%%
  Now we turn to give the proof for the second conclusion.
  Suppose $H_{\Phi}^* : W^{q,p} \rightarrow L^q$ is bounded.
  As in the proof of conclusion (1), we take $g_{N}(x)=\sum\limits_{j=1}^{N}\rho_j(x)\cdot|x|^{-n/q}$.
  A direction calculation yields that
  \begin{equation*}
    \begin{split}
      \left\| H_{\Phi}^*g_N \right\|_{L^q}
      &=
      \left\| \int_{\mathbb{R}^n}\Phi(y)|y|^ng_N(|y|x)dy \right\|_{L^q}
      \\
      &=
      \left\| \int_{\mathbb{R}^n}\Phi(y)|y|^{n/q'} \cdot \sum_{j=1}^{N}\rho_j(|y|x) \cdot |x|^{-n/q}dy \right\|_{L^p}
      \\
      &\geqslant
      \left\| \int_{B(0,4/3\cdot2^{M})\setminus B(0,3/2\cdot2^{-M})}\Phi(y)|y|^{n/q'} \cdot \sum_{j=1}^{N}\rho_j(|y| x) \cdot |x|^{-n/q}dy \right\|_{L^q}
      \\
      &\geqslant
      \int_{B(0,4/3\cdot2^{M})\setminus B(0,3/2\cdot2^{-M})}\Phi(y)|y|^{n/q'}dy
      \cdot \left\| |x|^{-n/q}\chi_{\{2^M<|x|<\cdot2^{N-M}\}}(x) \right\|_{L^q}
      \\
      &\gtrsim
      \int_{B(0,4/3\cdot2^{M})\setminus B(0,3/2\cdot2^{-M})}\Phi(y)|y|^{n/q'}dy \cdot (\lg2^{N-2M})^{1/q}.
    \end{split}
  \end{equation*}
  On the other hand,
  \begin{equation*}
    \begin{split}
      \|\mathscr{F}^{-1}(\sigma_kg_N)\|_{L^p}
      = &
      \|\mathscr{F}^{-1}(\sigma_k\sum\limits_{j=1}^{N}\rho_j(x)\cdot|x|^{-n/q})\|_{L^p}
      %\\
      \lesssim %&
      \langle k\rangle^{-n/q}\|\mathscr{F}^{-1}(\sigma_k\sum\limits_{j=1}^{N}\rho_j(x))\|_{L^p}
      \\
      \leqslant &
      \langle k\rangle^{-n/q}\|\mathscr{F}^{-1}\sigma_k\|_{L^p}\cdot
      \sum\limits_{1\leqslant j\leqslant N: \sigma_k\rho_j\neq 0}\|\mathscr{F}^{-1}(\rho_j(x))\|_{L^1}
      \\
      \lesssim &
      \langle k\rangle^{-n/q}.
    \end{split}
  \end{equation*}

Using Lemma \ref{lemma, relation of Mpp Wpp}, we obtain that
  \begin{equation*}
    \begin{split}
      \left\| g_N \right\|_{W_{q,p}}
      &=
      \left\| \mathscr{F}^{-1}g_N \right\|_{M_{p,q}}
      =
      \left(\sum_{k\in\mathbb{Z}^n}\left\| \mathscr{F}^{-1}(\sigma_k g_N) \right\|_{L^p}^q\right)^{1/q}
      \\
      &\lesssim
      \left(
      \sum_{\substack{|k|<2^{N+1}\\k\in\mathbb{Z}^n}} \langle k\rangle^{-n}
      \right)^{1/q}
      \sim
      \left(\lg2^{N}\right)^{1/q}.
    \end{split}
  \end{equation*}
  We deduce that
  \begin{equation*}
    \|H_{\Phi}^*\|_{W_{q,p}\rightarrow L^q}
    \geqslant
    \frac{\left\| H_{\Phi}^*f_N\right\|_{L^q}}{\|f_N\|_{W_{q,p}}}
    \gtrsim
    \int_{B(0,\frac{4}{3}\cdot2^{M})\setminus B(0,\frac{3}{2}\cdot2^{-M})}\Phi(y)|y|^{n/q'}dy
    \left(\frac{\lg2^{N-2M}}{\lg2^{N}}\right)^{1/q}.
  \end{equation*}
  Letting $N\rightarrow\infty$, we have
  \begin{equation*}
    \int_{B(0,\frac{4}{3}\cdot2^{M})\setminus B(0,\frac{3}{2}\cdot2^{-M})}\Phi(y)|y|^{n/q'}dy
    \lesssim
    \|H_{\Phi}^*\|_{W_{q,p}\rightarrow L^q}.
  \end{equation*}
  By the arbitrariness of $M$, we let $M\rightarrow \infty$ and obtain that
  $\int_{\mathbb{R}^n}\Phi(y)|y|^{n/q'}dy\lesssim \|H^*_{\Phi}\|_{W_{q,p}\rightarrow L^q}$.
\end{proof}

Next, we establish the
following two propositions for reduction.
\begin{proposition}[For reduction of modulation space]\label{proposition, for reduction of modulation space}
  Let $1/2\leqslant 1/p \leqslant 1/q\leqslant 1$, $\Phi$ be a nonnegative function satisfying (\ref{basic assumption for Phi}).
  If the Hausdorff operator $H_{\Phi}$ is bounded on $M_{p,q}$, we have
  \begin{enumerate}
    \item $H_{\Phi} : M_{p,q} \rightarrow L^p$ is bounded,
    \item $H_{\Phi}^* :W_{q,p} \rightarrow L^q$ is bounded.
  \end{enumerate}
\end{proposition}
\begin{proof}
  The first conclusion can be deduced by the embedding relation
  $M_{p,q}\hookrightarrow L^p$ (see Lemma \ref{lemma, embedding relations between Mpp and Lp}) directly.
  We turn to prove the second conclusion.
  For any Schwartz function $f$, by the property of $H_{\Phi}$ and Lemma \ref{lemma, relation of Mpp Wpp}, we have
  $\|f\|_{M_{p,q}}=\|\widehat{f}\|_{W_{q,p}}$ and
  \begin{equation*}
    \|H_{\Phi}f\|_{M_{p,q}}=\|\widehat{H_{\Phi}f}\|_{W_{q,p}}
    =\|\widetilde{H_{\Phi}}\widehat{f}\|_{W_{q,p}}=\|H^*_{\Phi}\widehat{f}\|_{W_{q,p}}.
  \end{equation*}
  Thus, if $H_{\Phi}$ is bounded on $M_{p,q}$, we have
  \begin{equation*}
    \|H^*_{\Phi}\widehat{f}\|_{W_{q,p}}\lesssim \|\widehat{f}\|_{W_{q,p}}.
  \end{equation*}
  The embedding relation $W_{q,p}\hookrightarrow L^q$ then yields that
  \begin{equation*}
    \|H^*_{\Phi}f\|_{L^q}\lesssim \|f\|_{W_{q,p}}
  \end{equation*}
  for all $f\in \mathscr{S}(\mathbb{R}^n)$.
\end{proof}

\begin{proposition}[For reduction of Wiener amalgam space]\label{proposition, for reduction of Wiener amalgam space}
  Let $1/2\leqslant 1/q \leqslant 1/p\leqslant 1$, $\Phi$ be a nonnegative function satisfying (\ref{basic assumption for Phi}).
  If the Hausdorff operator $H_{\Phi}$ is bounded on $W_{p,q}$, we have
  \begin{enumerate}
    \item $H_{\Phi} : W_{p,q} \rightarrow L^p$ is bounded,
    \item $H_{\Phi}^* :M_{q,p} \rightarrow L^q$ is bounded.
  \end{enumerate}
\end{proposition}
\begin{proof}
  The first conclusion can be deduced by the embedding relation
  $W_{p,q}\hookrightarrow L^p$ (see Lemma \ref{lemma, embedding relations between Wpp and Lp}) directly.
  We turn to prove the second conclusion.
  For any Schwartz function $f$, by the property of $H_{\Phi}$ and Lemma \ref{lemma, relation of Mpp Wpp}, we have that
  $\|f\|_{W_{p,q}}=\|\widehat{f}\|_{M_{q,p}}$ and
  \begin{equation*}
    \|H_{\Phi}f\|_{W_{p,q}}=\|\widehat{H_{\Phi}f}\|_{M_{q,p}}
    =\|\widetilde{H_{\Phi}}\widehat{f}\|_{M_{q,p}}=\|H^*_{\Phi}\widehat{f}\|_{M_{q,p}}.
  \end{equation*}
  Thus, if $H_{\Phi}$ is bounded on $W_{p,q}$, we have
  \begin{equation*}
    \|H^*_{\Phi}\widehat{f}\|_{M_{q,p}}\lesssim \|\widehat{f}\|_{M_{q,p}}.
  \end{equation*}
  The embedding relation $M_{q,p}\hookrightarrow L^q$ then yields that
  \begin{equation*}
    \|H^*_{\Phi}f\|_{L^q}\lesssim \|f\|_{M_{q,p}}
  \end{equation*}
  for all $f\in \mathscr{S}(\mathbb{R}^n)$.
\end{proof}
Now, we are ready to give the proof of Theorem \ref{theorem, boundedness of Hausdorff operator on modulation space}.

\textbf{Proof of Theorem \ref{theorem, boundedness of Hausdorff operator on modulation space}.}
We divide this proof into two parts.

\textbf{IF PART:}
Using the Minkowski inequality, we deduce that
\begin{equation*}
\begin{split}
  \|H_{\Phi}f\|_{M_{p,q}}
  \lesssim &
  \left\| \int_{\mathbb{R}^n}\Phi(y)f(x/{|y|})dy \right\|_{M_{p,q}}
  \\
  \lesssim &\int_{\mathbb{R}^n}\Phi(y)\left\| f(x/{|y|}) \right\|_{M_{p,q}} dy.
\end{split}
\end{equation*}
Recalling the dilation properties of modulation space (see Lemma \ref{lemma, dilation property of modulation space}), we obtain that
\begin{equation*}
  \begin{split}
    \|H_{\Phi}f\|_{M_{p,q}}
     &\lesssim
              \int_{\mathbb{R}^n}\Phi(y) \max\{|y|^{n/p},|y|^{n/q'}\} dy\left\| f\right\|_{M_{p,q}}
     \\
     &\lesssim \int_{\mathbb{R}^n} \left( |y|^{n/{p}}+|y|^{n/q'} \right) \Phi(y) dy \left\| f\right\|_{M_{p,q}}.
  \end{split}
\end{equation*}
This implies the boundedness of $H_{\Phi}$ on $M_{p,q}$.% when $1/p<1/2$.

\textbf{ONLY IF PART:}
Suppose $H_{\Phi}$ is bounded on $M_{p,q}$.
If $1/2\leqslant 1/p\leqslant 1/q\leqslant 1$, the conclusion can be verified directly by Proposition \ref{proposition, for technique} and \ref{proposition, for reduction of modulation space}.

We only need to deal with the case of $1/q\leqslant 1/p\leqslant 1/2$.
We use a dual argument to deal with this case.
Recalling
\begin{equation*}
  \langle H_{\Phi}^*f|\ g\rangle=\langle f|\ H_{\Phi}g\rangle
\end{equation*}
for all $f,g\in \mathscr{S}(\mathbb{R}^n)$, by Lemma \ref{Holder inequality of modulation space and Wiener space} and Lemma \ref{lemma, embedding relations between Mpp and Lp} we deduce that
\begin{equation*}
  \begin{split}
    |\langle H_{\Phi}^*f|\ g\rangle|
    = &
    |\langle f|\ H_{\Phi}g\rangle|
    \\
    \leqslant &
    \|f\|_{M_{p',q'}}\|H_{\Phi}g\|_{M_{p,q}}
    \\
    \lesssim &
    \|f\|_{M_{p',q'}}\|g\|_{M_{p,q}}
        \\
    \lesssim &
    \|f\|_{M_{p',q'}}\|g\|_{L^p},
  \end{split}
\end{equation*}
which implies that
\begin{equation}\label{for proof, 1}
  \|H^*_{\Phi}f\|_{L^{p'}}\lesssim \|f\|_{M_{p',q'}}
\end{equation}
for all $f\in \mathscr{S}(\mathbb{R}^n)$.
In addition, by the boundedness of $H_{\Phi}$ on $M_{p,q}$, we use Lemma \ref{lemma, relation of Mpp Wpp} to deduce that
$H^*_{\Phi}$ is also bounded on $W_{q,p}$.
Thus, by Lemma \ref{Holder inequality of modulation space and Wiener space} and Lemma \ref{lemma, embedding relations between Wpp and Lp} we have
\begin{equation*}
  \begin{split}
    |\langle H_{\Phi}f|\ g\rangle|
    = &
    |\langle f|\ H_{\Phi}^*g\rangle|
    \\
    \leqslant &
    \|f\|_{W_{q',p'}}\|H_{\Phi}^*g\|_{W_{q,p}}
    \\
    \lesssim &
    \|f\|_{W_{q',p'}}\|g\|_{W_{q,p}}
        \\
    \lesssim &
    \|f\|_{W_{q',p'}}\|g\|_{L^q},
  \end{split}
\end{equation*}
which implies that
\begin{equation}\label{for proof, 2}
  \|H_{\Phi}f\|_{L^{q'}}\lesssim \|f\|_{W_{q',p'}}
\end{equation}
for all $f\in \mathscr{S}(\mathbb{R}^n)$.

Combining (\ref{for proof, 1}) and (\ref{for proof, 2}), observing $1/2\leqslant 1/p'\leqslant 1/q'$,
we use Proposition \ref{proposition, for technique} to get the conclusion.

\textbf{Proof of Theorem \ref{theorem, boundedness of Hausdorff operator on Wiener amalgam space}.}
We divide this proof into two parts.

\textbf{IF PART:}
Using Lemma \ref{lemma, relation of Mpp Wpp} and the Minkowski inequality, we deduce that
\begin{equation*}
\begin{split}
  \|H_{\Phi}f\|_{W_{p,q}}
  \sim&
  \|\widehat{H_{\Phi}f}\|_{M_{q,p}}
  \sim
  \|H_{\Phi}^{\ast}\widehat{f}\|_{M_{q,p}}
  %\\
  \sim
  \left\| \int_{\mathbb{R}^n}\Phi(y)|y|^n\widehat{f}(|y|x)dy \right\|_{M_{q,p}}
  \\
  \lesssim &\int_{\mathbb{R}^n}\Phi(y)|y|^n\left\| \widehat{f}(|y|x) \right\|_{M_{q,p}} dy.
\end{split}
\end{equation*}
Recalling the dilation properties of modulation space (see Lemma \ref{lemma, dilation property of modulation space}), we obtain that
\begin{equation*}
  \begin{split}
    \|H_{\Phi}f\|_{W_{p,q}}
     &\lesssim
              \int_{\mathbb{R}^n}\Phi(y)|y|^n \max\{|y|^{-n/q},|y|^{-n/p'}\} dy\|\widehat{f}\|_{M_{q,p}}
     \\
     &\lesssim \int_{\mathbb{R}^n} \left( |y|^{n/{p}}+|y|^{n/q'} \right) \Phi(y) dy \left\| f\right\|_{W_{p,q}}.
  \end{split}
\end{equation*}
This implies the boundedness of $H_{\Phi}$ on $W_{p,q}$.% when $1/p<1/2$.

\textbf{ONLY IF PART:}
Suppose $H_{\Phi}$ is bounded on $W_{p,q}$.
If $1/2\leqslant 1/q\leqslant 1/p\leqslant 1$, the conclusion can be verified directly by Proposition \ref{proposition, for technique} and \ref{proposition, for reduction of Wiener amalgam space}.

For the case $1/p\leqslant 1/q\leqslant 1/2$, the desired conclusion follows by a dual argument as in the proof of Theorem \ref{theorem, boundedness of Hausdorff operator on modulation space}.

\begin{remark}
  For some technical reasons,
  our main theorems only characterize the boundedness of Huasdorff operator on $M_{p,q}$ and $W_{p,q}$ in some special cases.
  Our theorems remain an open problem for the characterization of Huasdorff operator
  on the full range $1\leqslant p,q \leqslant \infty$.
\end{remark}
~\\
\\
{\bf Acknowledgments.}
The authors would like to express their deep thanks to the referee for many helpful comments and
suggestions.
This work is supported by the National Natural Science Foundation of China 
(No. 11601456, 11771388, 11371316, 11701112, 11671414) and China postdoctoral Science Foundation (No. 2017M612628).

\bibliographystyle{amsplain}

\end{document}